\documentclass[11pt]{amsart}
\usepackage[T1]{fontenc}
\usepackage{accanthis}
\usepackage[margin=1.35in]{geometry}
\usepackage{amscd,amsmath,amsxtra,amsthm,amssymb,stmaryrd,xr,mathrsfs,mathtools,enumerate,commath, comment}
\usepackage{stmaryrd}
\usepackage{xcolor}
\usepackage{commath}
\usepackage{comment}
\usepackage{tikz-cd}
\usepackage{float}
\usepackage{longtable} 
\usepackage{pdflscape} 
\usepackage{orcidlink}
\usepackage{booktabs}
\usepackage{hyperref}
\definecolor{vegasgold}{rgb}{0.77, 0.7, 0.35}
\definecolor{darkgoldenrod}{rgb}{0.72, 0.53, 0.04}
\definecolor{gold(metallic)}{rgb}{0.83, 0.69, 0.22}
\hypersetup{
 colorlinks=true,
 linkcolor=darkgoldenrod,
 filecolor=brown,      
 urlcolor=gold(metallic),
 citecolor=darkgoldenrod,
 pdftitle={On the distribution of eigenvalues in families of cayley graphs},
 }

\usepackage[all,cmtip]{xy}

\usepackage{tikz}
\usetikzlibrary{shapes.geometric}
\tikzset{every loop/.style={min distance=10mm,looseness=10}}

\DeclareFontFamily{U}{wncy}{}
\DeclareFontShape{U}{wncy}{m}{n}{<->wncyr10}{}
\DeclareSymbolFont{mcy}{U}{wncy}{m}{n}
\DeclareMathSymbol{\Sh}{\mathord}{mcy}{"58}
\usepackage[T2A,T1]{fontenc}
\usepackage[OT2,T1]{fontenc}

\newtheorem{theorem}{Theorem}[section]
\newtheorem{lemma}[theorem]{Lemma}

\newtheorem{proposition}[theorem]{Proposition}
\newtheorem{corollary}[theorem]{Corollary}

\numberwithin{equation}{section}

\theoremstyle{remark}
\newtheorem{remark}[theorem]{Remark}

\newcommand{\op}[1]{\operatorname{#1}}

\newcommand{\bE}{\mathbf{E}}

\newcommand{\cF}{\mathcal{F}}

\newcommand{\Z}{\mathbb{Z}}

\begin{document}
\title[Distribution of eigenvalues in families of Cayley graphs]{On the distribution of eigenvalues in families of Cayley graphs}

\author[M.~Lalin]{Matilde Lalin\, \orcidlink{0000-0002-7155-5439}}
\address[Lalin]{Universit\'e de Montr\'eal, Pavillon Andr\'e-Aisenstadt, D\'ept. de mathématiques et de statistique,
CP 6128, succ. Centre-ville,
Montr\'eal, Qu\'ebec, H3C 3J7, Canada}
\email{matilde.lalin@umontreal.ca}

\author[A.~Ray]{Anwesh Ray\, \orcidlink{0000-0001-6946-1559}}
\address[Ray]{Chennai Mathematical Institute, H1, SIPCOT IT Park, Kelambakkam, Siruseri,
Tamil Nadu 603103, India}
\email{anwesh@cmi.ac.in}

\begin{abstract}
We consider the family of undirected Cayley graphs associated with odd cyclic groups, and study statistics for the eigenvalues in their spectra. Our results are motivated by analogies between arithmetic geometry and graph theory. 
\end{abstract}

\date{\today} 

\maketitle
\section{Introduction}

It is well known that there are deep analogies between the arithmetic of curves in positive characteristic and the combinatorial properties of graphs. Ihara \cite{Ihara:1966} and Sunada \cite{Sunada:1986} defined the zeta function associated to a locally finite connected graph. The definition closely resembles the \emph{Selberg zeta function} and encodes asymptotics for closed walks in the graph. There is a close analogy between the zeta functions of curves in positive characteristic and the Ihara zeta functions associated to graphs. Further analogies have been realized by Baker and Norine \cite{Baker/Norine:2009}, who etablished an analogue of the Abel-Jacobi map and the Riemann-Roch theorem in graph theory. Recently, statistical questions have been framed and studied for curves varying in certain naturally occurring ensembles, see for instance, \cite{rudnick2008traces, kurlberg2009fluctuations, xiong2010fluctuations,bucur2011biased,entin2012distribution, thorne2014distribution, cheong2015distribution, bucur2016distribution, bucur2016statistics,sankar, ray2022statistics}. There has also been a growing interest in graph theoretic questions that are motivated by arithmetic statistics. The sandpile group (also known as the Jacobian group, the critical group, or the Picard group) is an abelian group naturally associated to a graph that codifies structural information of the graph. The statistics of such groups have been studied by various authors \cite{CKLPW, Wood}. More recently, the arithmetic statistics in the context of the Iwasawa theory of graphs has been studied in \cite{DLRV}.
\par It is natural to study statistical questions for the poles of the Ihara zeta function \cite{Ihara:1966,Bass,Terras}. For a finite regular graph, Ihara related the poles of the zeta function to the spectrum of the adjacency matrix. The spacing between the eigenvalues of regular graphs has been studied via numerical computation, and related to the Gaussian orthogonal ensemble, cf. \cite{JMRR, Newland}. Cayley graphs are central objects in the combinatorial and geometric group theory. They allow us to visualize a group with respect to a generating set. Its spectrum can be described very concretely \cite{Valette}. The main focus of this work is to study distribution questions for the family of eigenvalues that arise from Cayley graphs of finite odd cyclic groups.

\subsection{Main result} We now state our main result. Let $r\in \Z_{\geq 1}$ and let $\mathcal{F}$ be an infinite family of $r$-regular graphs with real eigenvalues. A \emph{counting function} $h: \mathcal{F}\rightarrow \Z_{\geq 1}$ is a function such that for any $x\in \mathbb{R}_{\geq 0}$, there are only finitely many graphs $X\in \mathcal{F}$ that satisfy $h(X)\leq x$. Note that any eigenvalue of an $r$-regular graph in $\mathcal{F}$ lies in the interval $[-r, r]$. Given an interval $J=[a,b]$ contained in $[-r, r]$, we wish to evaluate the probability that a random eigenvalue of a graph in $\mathcal{F}$ does lie in $J$. We consider families of Cayley graphs associated with cyclic groups with odd order. Given an odd natural number $n$, let $C_n$ denote the cyclic group with $n$ elements, which we identify with $\Z/n\Z$. For a pair $\gamma=(G, S)$ such that $S=S^{-1}$, recall that $X_\gamma$ is the associated Cayley graph. The association $\gamma\mapsto X_\gamma$ gives us a way to parametrize Cayley graphs. In greater detail, we set 
\[\cF_{r}:=\left\{X_\gamma\mid \gamma=(G, S); G=C_{n}\text{ for some odd number } n>1, \#S=r, S=S^{-1}\right\}.\]
Thus, $\cF_{r}$ is an infinite family of $r$-regular graphs; define a height function by setting $h(X_\gamma):=n$, where $\gamma=(C_n, S)$. There are only finitely many Cayley graphs for a given group, and hence, $h$ is a counting function. 

\par For $x\geq 0$, let $\sqrt{x}$ be the non-negative square root of $x$. Set $\delta$ denote the normalized arcsine distribution \[\delta(u):=\begin{cases} \frac{1}{\pi\sqrt{1-u^2}} & \text{ if }u\in (-1,1);\\
0 & \text{ otherwise;}\\
\end{cases}\]
and let $\delta^{(k)}(u)$ be the $k$-fold convolution 
\[\delta^{(k)}(u):=\int_{-1}^1\dots \int_{-1}^1 \delta(u_1)\delta(u_2)\cdots \delta(u_{k-1})\delta\left(u-\sum_{i=1}^{k-1} u_i\right) du_1\cdots du_{k-1}.\]
Given a graph $X$, let $\op{Sp}(X)$ denote its spectrum, i.e., the set of eigenvalues for its adjacency matrix.
\begin{theorem}\label{main thm}
    Let $r\in \Z_{\geq 2}$ and set $k:=\lfloor \frac{r}{2}\rfloor$. Let $J=[a, b]$ be an interval contained in $[-r, r]$. Set $[c,d]\subseteq [-k, k]$ to denote the interval $[a/2, b/2]$ (resp. $[(a-1)/2, (b-1)/2]$) if $r$ is even (resp. $r$ is odd). Then, as $n\rightarrow \infty$,
    \[\begin{split} \frac{\#\{(X, \alpha)\mid X\in \cF_r, \alpha\in \op{Sp}(X), h(X)=n, \alpha\in J\} }{\#\{(X, \alpha)\mid X\in \cF_r, \alpha\in \op{Sp}(X), h(X)=n\}}  
    =& \int_{c}^d \delta^{(k)}(u) du+ O_r(n^{-1+\epsilon}),\end{split}\]
    where the implicit constant in the error term only depends on $r$. 
\end{theorem}
The above theorem is a direct consequence of Theorem \ref{main theorem 5.5}. We note that the eigenvalues in the spectrum of a Cayley graph in $\cF_1$ are all equal to $1$, hence, we only consider the case when $r\geq 2$. The above probabilities are reinterpreted in terms of certain lattice point counts in certain regions in Euclidean space. These lattice point counts are estimated via a combination combinatorial and analytic techniques. 

\par The above result motivates the study of eigenvalue distributions in more general families of $r$-regular Cayley graphs, ordered by the size of the underlying group. The authors expect that the study of such questions shall lead to many interesting developments in the future.

\subsection*{Acknowledgements} The authors would like to thank the anonymous referee for carefully checking the manuscript, and Daniel Valli\`eres for helpful discussions. This work is supported by  the Natural Sciences and Engineering Research Council of Canada, Discovery Grant 355412-2022, the Fonds de recherche du Qu\'ebec - Nature et technologies, Projet de recherche en \'equipe 300951,  the Simons Foundation, and the Centre de recherches math\'ematiques. 

\section*{Statements and Declarations}

\subsection*{Competing Interests} The authors report there are no competing interests to declare.

\subsection*{Data Availability} There is no data associated to the results of this manuscript.  

\section{Preliminaries}
\subsection{Basic definitions and properties of graphs}
\par In this section, we recall a few preliminary notions in graph theory. For a more detailed exposition, we refer to \cite{Sunada:2013}. The data of a directed graph $X$ consists of a tuple $(V_X, \mathbf{E}_X)$, where $V_X$ is the set of vertices and $\mathbf{E}_X$ is the set of directed edges. The map $\op{inc}: \bE_X\rightarrow V_X\times V_X$ is the incidence map. The source and target maps $o: \bE_X\rightarrow V_X$ and $t:\bE_X\rightarrow V_X$ send a directed edge to the source and target vertices respectively, and $\op{inc}(e)=(o(e), t(e))$. Given $e\in \bE_X$, we say that $e$ joins $o(e)$ to $t(e)$. We shall make the following assumptions on our graphs $X$.
\begin{enumerate}
    \item The sets $V_X$ and $\bE_X$ are finite.
    \item Given vertices $v_1\in V_X$ and $v_2\in V_X$ (not necessarily distinct), there is at most one edge joining $v_1$ to $v_2$. 
    \item Given any edge $e\in \bE_X$, there is an edge $\bar{e}$ such that $o(\bar{e})=t(e)$ and $t(\bar{e})=o(e)$. Note that if $o(e)=t(e)$ (i.e., $e$ is a loop), then $\bar{e}=e$. 
\end{enumerate}
Consider an equivalence relation on the set $\bE_X$ of directed edges, identifying $e$ with $\bar{e}$. The equivalence classes consist of undirected edges of the graph and are denoted by $E_X$. Given $v\in V_X$, we consider the set of directed edges that orginate at $v$, 
\[\bE_{X, v}:=\{e\in \bE_X\mid o(e)=v\}.\] The valency of $v$ is defined as follows $\op{val}_X(v):=\# \bE_{X,v}$. Alternatively, the valency is also the number of undirected edges which contain $v$ as one of its vertices. The graph is $r$-regular if $\op{val}_X(v)=r$ for all $v\in V_X$. Let $g:=\# V_X$ and choose a labelling on the vertices $V_X=\{v_1, \dots, v_g\}$. Consider the $g\times g$ diagonal matrix defined by 
\[d_{i, j}=\begin{cases} & 0 \text{ if }i\neq j;\\
& \op{val}_X(v_i) \text{ if }i=j.\\
\end{cases}\]
The adjacency matrix $A=(a_{i,j})$ is defined as follows 
\[a_{i,j}=\begin{cases} & \text{number of undirected edges between }v_i\text{ and }v_j \text{ if }i\neq j;\\
& \text{twice the number of undirected loops at }v_i\text{ if }i=j.\\
\end{cases}\]

The spectrum $\op{Sp}(X)$ is the set of eigenvalues $\alpha$ of the adjacency matrix $A$. The matrix $Q:=D-A$ is called the \emph{Laplacian matrix} of the graph $X$. Let us explain the explicit relationship between the eigenvalues of the adjacency matrix and the poles of \emph{Ihara zeta function} associated to $X$. A path in $X$ is a sequence $C=a_1,\dots ,a_k$, where $(a_1,\dots, a_k)$ is a tuple of directed edges of $X$ such that $t(a_i)=o(a_{i+1})$. The length of $C$ is the number of edges $\nu(C):=k$. A closed path is a path for which $o(a_1)=t(a_k)$, i.e. a path starting and ending at the same vertex. The path is said to have a \emph{backtrack} if $a_{j+1}=\overline{a_j}$ for some $j$ in the range $1\leq j \leq k-1$. The path is said to have a \emph{tail} if $a_k=\overline{a_1}$. Given a closed path $D$, and an integer $m>0$, $D^m$ is the path obtained by moving around $D$ $m$-times. A closed path $C$ is said to be \textit{primitive} if $C$ has no backtracks or tails and $C\neq D^m$ for any closed path $D$ and $m>1$. Two closed paths are equivalent if one can be obtained from the other by changing the starting vertex. Note that if $C$ is primitive, there are exactly $\nu(C)$ paths that are equivalent to $C$. A \emph{prime} path is an equivalence class of primitive paths. We shall sometimes denote the equivalence class of a primitive path $C$ by $[C]$, however, when there is no cause for confusion, we simply use $C$ itself to denote the prime path with primitive representative $C$. 

\par The \emph{Ihara polynomial} $h_X(u)$ denotes the determinant
\[h_X(u):=\det(I-Au+Qu^2).\] Suppose that $X$ is connected with no vertex of degree $1$.
With these definitions in mind, the Ihara zeta function is defined to be the product
\[\zeta_X(u)=\prod_C (1-u^{\nu(C)})^{-1},\]
where the product ranges over all primes in $X$, ordered according to length. Then the Ihara $3$-term determinant formula states that
\[\zeta_X(u)^{-1}=(1-u^2)^{r(X)-1}h_X(u),\]
where $r(X)$ is rank of the fundamental group, $r(X)=|E|-|V|+1$. Note that when $X$ is $r$-regular, where $r\geq 2$. Then $Q=(r-1)\op{Id}$, and 
\[h_X(u)=\prod_{\alpha\in \op{Sp}(X)}\left(1-\alpha u+(r-1)u^2\right).\]
Therefore, the roots of $h_X(u)$ are given by 
\[r_\alpha^{\pm} =\frac{\alpha\pm \sqrt{\alpha^2-4(r-1)}}{2(r-1)},\]
where $\alpha\in \op{Sp}(X)$.
\par Let $G$ be a finite group and $S$ be a subset of $G$ such that $S^{-1}:=\{s^{-1}\mid s\in S\}$ is equal to $S$. Associated to the pair $\gamma=(G, S)$ is the undirected Cayley graph $X_\gamma$. The vertices $V_\gamma:=V_{X_\gamma}$ is the set $\{v_g\mid g\in G\}$, and there is an undirected edge joining $v_g$ to $v_h$ if and only if $gh^{-1}$ is contained in $S$. Consider the special case when $G$ is a finite abelian group and $\widehat{G}:=\op{Hom}(G, \mathbb{C}^\times)$ be the group of characters defined on $G$. A Cayley graph $X_\gamma$ is $r$-regular, where, $r:=\# S$. Given a character, $\chi\in \widehat{G}$, set $\lambda_\chi(X_\gamma):=\sum_{s\in S} \chi(s)$. 
As $\chi$ ranges over $\widehat{G}$, the eigenvalues $\lambda_\chi(X_\gamma)$ are all distinct (see Section 1.4.9 on Cayley graphs in \cite{Brouwer-Haemers}). 
The spectrum $\op{Sp}(X_\gamma)$ is the set of eigenvalues of the adjacency matrix and is equal to the set $\{\lambda_\chi(X_\gamma)\mid \chi \in \widehat{G}\}$. Since $S=S^{-1}$, we find that the eigenvalues $\lambda_\chi(X_\gamma)$ are all real and contained in $[-r, r]$.

\section{Eigenvalues of Cayley graphs}

Let $r\in \Z_{\geq 2}$, and recall that $\cF_r$ is the family of Cayley graphs 
\[\cF_{r}:=\left\{X_\gamma\mid \gamma=(G, S); G=C_{n}\text{ for some odd number } n>1, \#S=r, S=-S\right\},\]
where we are writing $C_n$ with additive notation, since we identify it with the group $\Z/n\Z$.

Since $n$ is assumed to be odd, we find that $r=\#S$ is odd precisely when $S$ contains $0$. In this case, we write,  
\[\begin{cases}
r=2k & \text{ if }r\text{ is even};\\
r=2k+1 & \text{ if }r\text{ is odd.}\\
\end{cases}\]
In either case, we may then write $S\backslash\{0\}=T\cup \left(-T\right)$, where $T=\{a_1, \dots, a_k\}$ and $1\leq a_i\leq \left( \frac{n-1}{2} \right)$ for $i=1, \dots, k$. We shall assume without loss of generality that $1\leq a_1<a_2<\dots <a_k\leq \left( \frac{n-1}{2} \right)$. 

Let $S$ be a subset of $G=C_n$ such that $\#S=r$ and $S=-S$. We set $\gamma:=(G, S)$. 
Recall that the spectrum of $X_\gamma$ is given by $\op{Sp} (X_\gamma)=\{\lambda_{\chi}(X_\gamma)\mid \chi\in \widehat{G}\}$.
For an integer $m\in [0, n-1]$, let $\chi_m\in \widehat{C_{n}}$ be the character defined by $\chi_m(x):=\op{exp}\left(\frac{2\pi i m x}{n}\right)$. We set $\lambda_m(\gamma):=\lambda_{\chi_m}(X_\gamma)$ and we find that 
\begin{equation}\label{lambda_m equation}\lambda_m(\gamma)=\begin{cases}  \sum_{j=1}^k 2\op{cos}\left(\frac{2\pi  m a_j}{n}\right) &\text{ if }r\text{ is even};\\  1+\sum_{j=1}^k 2\op{cos}\left(\frac{2\pi  m a_j}{n}\right) &\text{ if } r \text{ is odd.}
\end{cases}\end{equation}
In particular, we note that $\lambda_0(\gamma)=r$ is the largest eigenvalue.

\par We set $A_k(n)$ to denote the set of all vectors $\vec{a}=(a_1, \dots, a_k)\in \Z^k$ such that $1\leq a_1<a_2<\dots< a_k\leq \left(\frac{n-1}{2}\right)$. Given $\vec{a}\in A_k(n)$, let $\gamma_{\vec{a}}:=(C_n, S_{\vec{a}})$, where $S_{\vec{a}}$ is the set with $r$ elements such that $S\backslash \{0\}=\{a_1, -a_1, a_2, -a_2, \dots, a_r, -a_r\}$. We set
\[\tau_m(\vec{a})=\tau_m(\gamma_{\vec{a}}):=\sum_{i=1}^k\op{cos}\left(\frac{2\pi m a_i}{n}\right).\]
Therefore, from \eqref{lambda_m equation}, we find that
\[\lambda_m(\vec{a})=\lambda_m(\gamma_{\vec{a}})=\begin{cases}
 2\tau_m(\vec{a}) &\text{ if }r\text{ is even};\\  1+ 2\tau_m(\vec{a}) &\text{ if } r \text{ is odd.}
\end{cases}\]
We find that $\tau_m(\vec{a})\in [-k, k]$, and that $\lambda_m(\vec{a})\in [-r, r]$. 

Let $\mathcal{S}(n,k)$ be the set of pairs $\left(\gamma_{\vec{a}}, \tau_m(\vec{a})\right)$, where $\vec{a}\in A_k(n)$, and $m\in [0, n-1]$. 
Let $\mathcal{S}(n,k,m)$ be the subset of $\mathcal{S}(n, k)$ where $m$ is fixed.  Given and interval $[e,f]\subset \mathbb{R}$, we set $[e,f]_{\Z}:=\Z\cap [e,f]$. We identify $\mathcal{S}(n,k)$ with $A_k(n)\times \left[0, n-1\right]_{\Z}$. Moreover, we identify $\mathcal{S}(n,k,m)$ with $A_k(n)$, upon identifying $\vec{a}\in A_k(n)$ with $(\vec{a},m)\in A_k(n)\times \left[0, n-1\right]_{\Z}$. We refer to $\mathcal{S}(n,k,m)$ as the \emph{$m$-slice} in $\mathcal{S}(n,k)$. Let $J=[a,b]\subseteq [-r, r]$ be a closed interval and $I=[c,d]\subseteq [-k, k]$ defined by
\[[c,d]=\begin{cases}
[a/2, b/2] &\text{ if }r\text{ is even};\\
[(a-1)/2, (b-1)/2] &\text{ if }r\text{ is odd}.
\end{cases}\]
Note that $\lambda_m(\vec{a})\in J$ if and only if $\tau_m(\vec{a})\in I$.
Set $\mathcal{S}_I(n,k)\subseteq \mathcal{S}(n,k)$ (resp. $\mathcal{S}_I(n,k,m)\subseteq \mathcal{S}(n,k,m )$) to denote the subset for which $\tau_m(\vec{a})\in I$. 

\par We have the following identifications
\begin{equation}\label{identifications}\begin{split}
&\mathcal{S}(n,k)=\{(X, \alpha)\mid X\in \cF_r, \alpha\in \op{Sp}(X), h(X)=n\}, \\
&\mathcal{S}_I(n,k)=\{(X, \alpha)\mid X\in \cF_r, \alpha\in \op{Sp}(X), h(X)=n, \alpha\in J\}. \\
\end{split}\end{equation}
It is easy to see that
\[\begin{split}
& \#\mathcal{S}(n, k, m)=\binom{\frac{n-1}{2}}{k},\\
& \#\mathcal{S}(n, k)=n\binom{\frac{n-1}{2}}{k}.\\
\end{split}\]
For $n, k\in \Z_{\geq 1}$ such that $n$ is odd and $k\leq \frac{n-1}{2}$, set
\[\op{Prob}_I\left(n, k\right):=\frac{\#\mathcal{S}_I(n, k)}{\#\mathcal{S}(n,k)}.\]
Note that after making the identifications \eqref{identifications}, we have that
\begin{equation}\label{probI}\op{Prob}_I\left(n, k\right)=\frac{\#\{(X, \alpha)\mid X\in \cF_r, \alpha\in \op{Sp}(X), h(X)=n,\text{ and }\alpha\in J\}}{\#\{(X, \alpha)\mid X\in \cF_r, \alpha\in \op{Sp}(X),h(X)=n\}}.\end{equation}

\par For a fixed integer $k\in \Z_{\geq 1}$, we wish to compute the limit 
\[\begin{split}\op{Prob}_I(k):=& \lim_{n \rightarrow  \infty} \op{Prob}_I\left(n, k\right),\\ 
=& \lim_{n\rightarrow \infty} \left(\frac{ (2^k k!) \#\mathcal{S}_I(n, k)}{n^{k+1}}\right).\end{split}\] 
We express $\op{Prob}_I(k)$ as an average
\begin{equation}\label{prob sum eqn}\op{Prob}_I(n, k)=\frac{1}{n}\sum_{m=0}^{n-1} \op{Prob}_I(n, k, m),\end{equation}
where 
\[\op{Prob}_I(n, k, m):=\frac{\#\mathcal{S}_I(n, k, m)}{\#\mathcal{S}(n, k, m)}=\frac{\#\mathcal{S}_I(n, k, m)}{\binom{\frac{n-1}{2}}{k}}.\]

In the next section, we shall set up a geometric interpretation for the probability $\op{Prob}_I(n, k, m)$. Let $B_k$ denote the $k$-dimensional box 
\[B_k:=\{(x_1, \dots, x_k)\mid 0\leq x_i< 1/2\text{ for all }i\},\] and note that $\op{Vol}(B_k)=2^{-k}$. Let $B_k(I)$ be the subset of $B_k$ defined by
\begin{equation}\label{def of B_k(I)}B_k(I):=\{(x_1, \dots, x_k)\in B_k\mid \sum_{i=1}^k \op{cos}(2\pi x_i)\in I\}.\end{equation} 
\par Let us compute $\op{Vol}\left(B_k(I)\right)$ for an interval $I=[c,d]$. Let $D_k(I):=\{(u_1, \dots, u_k)\in [-1,1]^k\mid \sum_i  u_i\in I\}$, and $\Phi$ be the function taking $D_k(I)$ to $B_k(I)$ defined by $\Phi(u_1, \dots, u_k)=\left(\frac{1}{2\pi }\op{cos}^{-1}(u_1), \dots, \frac{1}{2\pi }\op{cos}^{-1}(u_k)\right)$. For $x\geq 0$, let $\sqrt{x}$ be the non-negative square root of $x$. Recall that \[\delta(u):=\begin{cases} \frac{1}{\pi\sqrt{1-u^2}} & \text{ if }u\in (-1,1);\\
0 & \text{ otherwise;}\\
\end{cases}\]
and let $\delta^{(k)}(u)$ be the $k$-fold convolution 
\[\delta^{(k)}(u):=\int_{-1}^1\dots \int_{-1}^1 \delta(u_1)\delta(u_2)\cdots \delta(u_{k-1})\delta\left(u-\sum_{i=1}^{k-1} u_i\right) du_1\cdots du_{k-1}.\]

\begin{lemma} \label{lem:vol-computation}
Let $I=[c,d]\subseteq [-k,k]$ be an interval, $k\in \Z_{\geq 1}$, and let $B_k(I)\subset \mathbb{R}^k$ be the region defined by \eqref{def of B_k(I)}. Then, its volume is given by
\[\op{Vol}\left(B_k(I)\right)=2^{-k}\int_{c}^d \delta^{(k)}(u) du.\]
\end{lemma}
\begin{proof}
   By the change of variables theorem, we find that 
\begin{equation*}
\begin{split}\op{Vol}\left(B_k(I)\right)=& \int_{D_k(I)} |\op{det} D\Phi|\\ =& (2\pi)^{-k}\int_{D_k(I)} \frac{1}{\left(\sqrt{1-u_1^2}\sqrt{1-u_2^2}\dots \sqrt{1-u_r^2}\right)} du_1\dots d u_r \\
= & 2^{-k}\int_{c}^d \delta^{(k)}(u) du.
\end{split}\end{equation*}
\end{proof}

\section{A geometric interpretation}
\par We wish to provide a geometric reinterpretation for the quantities $\#\mathcal{S}_I(n, k, m)$ and $\#\mathcal{S}(n, k, m)$ introduced in the previous section. Given an integer $n\geq 1$, let $L_n\subset \mathbb{R}^k$ be the lattice consisting of all vectors of the form $(\frac{a_1}{n}, \dots, \frac{a_k}{n})$, where $(a_1, \dots, a_k)\in \Z^k$. Let $L_n'$ be the subset of $L_n$ consisting of such vectors $(\frac{a_1}{n}, \dots, \frac{a_k}{n})$ for which the coordinates are all mutually distinct. Set $\Omega_n:=B_k\cap L_n$, $\Omega_n(I):=B_k(I)\cap L_n$, $\Omega_n':=B_k\cap L_n'$ and $\Omega_n'(I):=B_k(I)\cap L_n'$ and let $\kappa:\mathbb{R}\rightarrow [0,1/2]$ be the function defined as follows
\begin{equation}\label{eq:omega}\kappa(x):=\begin{cases} \{x\} & \text{ if } \{x\}\leq 1/2;\\
1-\{x\} & \text{ if } \{x\}> 1/2.
\end{cases}\end{equation} Here, $\{x\}$ denotes the fractional part of $x$. We note that $\op{cos} (2\pi x)=\op{cos} \left(2 \pi \kappa(x)\right)$. For a pair $\gamma=(G, S)$, such that $G=C_{n}$ and $S$ a subset of $G$ with $S=S^{-1}$, we recall from the previous section that $S\backslash \{0\}=T\cup(-T)$, where $T=\{a_1, \dots, a_k\}$, and $\vec{a}=(a_1, \dots, a_k)\in A_k(n)$. Let $\sigma$ be a permutation of $\{1, \dots, k\}$ and $m\in [1, n-1]_{\Z}$. Write $m=dm_1$, $n=dn_1$, where $d:=(m, n)$. Consider a point $(\vec{a}, m)\in \mathcal{S}(n, k, m)$. Associate to this point and to the permutation $\sigma$, the point \begin{equation}\label{def of P}\begin{split}P(\vec{a}, \sigma, m, n):=& \left(\kappa\left(\frac{ma_{\sigma(1)}}{n}\right), \kappa\left(\frac{ma_{\sigma(2)}}{n}\right),\dots, \kappa\left(\frac{ma_{\sigma(k)}}{n}\right)\right),\\ 
=& \left(\kappa\left(\frac{m_1a_{\sigma(1)}}{n_1}\right), \kappa\left(\frac{m_1a_{\sigma(2)}}{n_1}\right),\dots, \kappa\left(\frac{m_1a_{\sigma(k)}}{n_1}\right)\right). \end{split}\end{equation} 
It is clear that $P(\vec{a}, \sigma, m, n)$ is contained in $\Omega_{n_1}$. 

\par For ease of notation, identify the $m$-slice $\mathcal{S}(n, k, m)$ with $A_k(n)$, by identifying $\vec{a}$ with $(\vec{a}, m)$.
For $m\in [0, n-1]_{\Z}$, it conveniences us to further subdivide $\mathcal{S}(n, k, m)$ into two disjoint subsets $\mathcal{S}'(n, k, m)$ and $\mathcal{S}''(n, k, m)$ as follows. Set $\mathcal{S}'(n, k, 0):=\emptyset$ and $\mathcal{S}''(n, k, 0):=\mathcal{S}(n, k, 0)$. For $m\in [1, n-1]_{\Z}$, let $\mathcal{S}'(n, k, m)$ consist of all $\vec{a}\in A_k(n)$ such that for $i<j$, 
\[a_i\not \equiv \pm a_j\mod{n_1},\]
and set $\mathcal{S}''(n, k, m)$ to be the complement $\mathcal{S}(n, k, m)\backslash \mathcal{S}'(n, k, m)$. We shall set $\mathcal{S}'(n, k)$ (resp. $\mathcal{S}''(n, k)$) to denote the disjoint union of slices $\mathcal{S}'(n, k, m)$ (resp. $\mathcal{S}''(n, k, m)$) as $m$ ranges through $[0,n-1]_{\Z}$. It shall turn out that the slices $\mathcal{S}'(n, k, m)$ are better suited to combinatorial arguments than $\mathcal{S}(n, k, m)$. Lemma \ref{bounding S'' lemma} will show that when $k$ is fixed, $\# \mathcal{S}''(n, k)$ is small in comparison to $\# \mathcal{S}(n, k)$, as $n$ goes to $\infty$. Given an interval $I\subseteq [-k, k]$, we set $\mathcal{S}_I'(n,k,m)$ (resp. $\mathcal{S}_I''(n,k,m)$) to be the subset of $\mathcal{S}'(n,k,m )$ (resp. $\mathcal{S}''(n,k,m)$) for which $\tau_m(\vec{a})\in I$.

\begin{lemma}\label{lemma 3.2}
    There is a constant $C_k>0$ (depending on $k$), independent of $m$ and $n$ such that
    \[\#\mathcal{S}''(n, k, m) < \frac{ C_k n^k}{n_1}.\] 
    The constant $C_k$ can be taken to be a polynomial in $k$.
\end{lemma}

\begin{proof}
    Given a pair $\tau=(i,j)$ such that $i<j$, let $\mathcal{S}''_\tau(n, k, m)$ be the subset of $\mathcal{S}''(n, k, m)$ such that \[a_i\equiv \pm a_j\mod{n_1}.\] The total number of subsets $\{a_1, \dots,a_{j-1}, a_{j+1}, \dots, a_k\}$ is at most $\binom{\left(\frac{n-1}{2}\right)}{k-1}$. Given one such choice, \[\begin{split}a_j\in & \{a_i+r n_1\mid -\lfloor n/n_1\rfloor-1\leq r \leq \lfloor n/n_1\rfloor+1\} \\
    \cup & \{-a_i+r n_1\mid -\lfloor n/n_1\rfloor-1\leq r \leq \lfloor n/n_1\rfloor+1\}.\end{split}\] Therefore, we find that 
    \[\#\mathcal{S}''_\tau(n, k, m)\leq 4\left(\lfloor n/n_1\rfloor+1\right)\binom{\left(\frac{n-1}{2}\right)}{k-1}.\]
    Thus, we have that 
    \[\#\mathcal{S}''(n, k, m)\leq 4\binom{k}{2}\left(\lfloor n/n_1\rfloor+1\right)\binom{\left(\frac{n-1}{2}\right)}{k-1}<\frac{C_kn^k}{n_1}.\]
\end{proof}

Given $n\in \Z_{\geq 1}$, let $d(n)$ denote the number of divisors of $n$. Note that for any $\epsilon>0$, $d(n)=O(n^\epsilon)$. Given non-negative functions $f$, $g$ and $h$, we write \[f=g+O_k(h)\] if there is a constant $C_k>0$ which depends only on $k$ such that 
\[|f-g|\leq C_k h.\]
 
\begin{lemma}\label{bounding S'' lemma}
    There is a constant $C_k>0$, independent of $n$, and depending on $k$, such that 
    \[\#\mathcal{S}''(n, k)< C_k n^{k}\left(\sum_{d|n} \frac{\varphi(d)}{d}\right)<C_kd(n) n^k, \] where $d(n)$ is the number of divisors of $n$. As a result, we find that for any $\epsilon>0$,
    \[\#\mathcal{S}''(n, k)=O_k(n^{k+\epsilon}),\]
    and that
    \[\frac{\#\mathcal{S}''(n, k)}{\#\mathcal{S}(n, k)}=O_k(n^{-1+\epsilon}).\]
\end{lemma}
\begin{proof}
\par Let $m\in [1, n-1]_{\Z}$, set $d=(m, n)$. Then, Lemma \ref{lemma 3.2} implies that \[\#\mathcal{S}''(n, k, m)<C_k\frac{n^k}{n_1}= C_k n^{k-1} d.\]
We find that 
    \[\begin{split}\#\mathcal{S}''(n, k) &= \#\mathcal{S}''(n, k, 0)+ \sum_{m=1}^{n-1} \#\mathcal{S}''(n, k, m),\\
    &=\#\mathcal{S}''(n, k, 0)+ \sum_{d|n} \sum_{(m, n)=d} \#\mathcal{S}''(n, k, m),\\
     &< n^k+ C_k \sum_{d|n} n^{k-1} d \#\{m\in [1, n-1]_{\Z}\mid (m, n)=d\},\\
     &= n^k+ C_k \sum_{d|n} n^{k-1} d \varphi\left(\frac{n}{d}\right),\\
     &= n^k+ C_k n^{k}\left(\sum_{d|n} \frac{\varphi(d)}{d}\right).\\
     \end{split}
\]
Replace $C_k$ above with $(C_k+1)$ to obtain the result. 
\end{proof}

\begin{remark}
The function $g(n):= \sum_{d|n} \varphi(d) \left(\frac{n}{d}\right)$ is known as  Pillai's arithmetical function \cite{Pillai:1933}. Set $\omega(n)$ to denote the number of prime divisors of $n$. We note that a bound of 
    \[g(n) \leq 27n \left(\frac{\log n}{\omega(n)} \right)^{\omega(n)}\]
    is explicitly proven by Broughan \cite[Theorem 3.1]{Broughan:2001}.
\end{remark}

\begin{lemma} \label{lem:interval}
Let $\sigma\in S_k$, $m\in [1, n-1]_{\Z}$, and $\vec{a}\in \mathcal{S}'(n, k, m)$. Let $I$ be an interval contained in $[-k,k]$. Then, the point $P(\vec{a}, \sigma, m, n)$ defined above in \eqref{def of P} lies in $\Omega_{n_1}'$. Furthermore, the following are equivalent
\begin{enumerate}
    \item $\tau_m(\vec{a})\in I$, 
    \item $P(\vec{a}, \sigma, m, n)\in \Omega_{n_1}'(I)$.
\end{enumerate}
\end{lemma}
\begin{proof}
Suppose that 
    \[\kappa\left(\frac{m_1a_{\sigma(i)}}{n_1}\right)=\kappa\left(\frac{m_1a_{\sigma(j)}}{n_1}\right),\] where we recall that $\kappa(x)$ is given by \eqref{eq:omega}. Then, 
    \[m_1 a_{\sigma(i)}\equiv \pm m_1 a_{\sigma(j)}\mod{n_1}.\] Since $m_1$ is coprime to $n_1$, we find that 
     \[a_{\sigma(i)}\equiv \pm  a_{\sigma(j)}\mod{n_1}.\]
     However, since $\vec{a}\in \mathcal{S}'(n, k, m)$,
     \[a_{\sigma(i)}\not\equiv \pm  a_{\sigma(j)}\mod{n_1}\] unless $i=j$. This proves that all the coordinates of the point $P(\vec{a}, \sigma, m, n)$ are distinct. Furthermore, all coordinates are of the form $b_i/n_1$, where $b_i$ is an integer, and these numbers lie in the range $[0, 1/2]$. Therefore, the point $P(\vec{a}, \sigma, m, n)$ is contained in $\Omega_{n_1}'$.

     \par Write $P(\vec{a}, \sigma, m, n)=(y_1, \dots, y_k)$, where, $y_i=\kappa\left(\frac{ma_{\sigma(i)}}{n}\right)$. Recall that $B_k(I)$ consists of all tuples $(x_1, \dots, x_k)\in B_k$ such that $\sum_{i=1}^k \op{cos}(2\pi x_i)\in I$. The point $(y_1, \dots, y_k)\in \Omega_{n_1}'$ lies in $\Omega_{n_1}'(I)$ if and only if $\sum_{i=1}^k \op{cos}(2\pi y_i)\in I$. Observe that 
     \[\begin{split}
     & \sum_{i=1}^k \op{cos}(2\pi y_i)\\
     = & \sum_{i=1}^k \op{cos}\left(2\pi \kappa\left(\frac{m_1a_{\sigma(i)}}{n_1}\right)\right) \\
      = & \sum_{i=1}^k \op{cos}\left( \frac{2\pi m_1a_{\sigma(i)}}{n_1}\right) \\
       = & \sum_{i=1}^k \op{cos}\left( \frac{2\pi m_1 a_{i}}{n_1}\right)=\tau_m(\vec{a}). \\
     \end{split}\]
     This means that the conditions are equivalent.
\end{proof}

\begin{lemma}\label{lemma 4.5}
    With respect to notation above,
    \[\begin{split} & \# \mathcal{S}'(n,k, m)=\left(\frac{(n,m)+1}{2}\right)^k \# \mathcal{S}'(n_1,k, m_1), \\ 
     & \# \mathcal{S}_I'(n,k, m)=\left(\frac{(n,m)+1}{2}\right)^k \# \mathcal{S}_I'(n_1,k, m_1). \\ \end{split}\]
\end{lemma}

\begin{proof}
    Given a tuple $(a_1, \dots, a_k)\in \mathcal{S}'(n_1,k, m_1)$ and $(a_1', \dots, a_k')\in \left[0, \left(\frac{(n,m)-1}{2}\right)\right]^k$, set $(b_1, \dots, b_k)$ be given by $b_i:=a_i+a_i'm_1$. Note that since $a_i\in [1, \frac{m_1-1}{2}]$, \[b_i<\left(\frac{m_1-1}{2}\right)+\left(\frac{(n,m)-1}{2}\right)m_1=\frac{m-1}{2}.\] Let $(c_1, \dots, c_k)$ be the permutation of the coordinates of $(b_1, \dots, b_k)$ such that $1\leq c_1<c_2<\dots < c_k\leq \frac{m-1}{2}$. Note that $a_i\not\equiv \pm a_j\mod{m_1}$ for $i\neq j$, and therefore, $c_i\not\equiv \pm c_j\mod{m_1}$ for $i\neq j$. This implies that $(c_1, \dots, c_k)\in \mathcal{S}'(n,k, m)$. The association 
    \[\left((a_1, \dots, a_k), (a_1', \dots, a_k')\right)\mapsto (c_1, \dots, c_k)\] sets up a bijection
    \[\mathcal{S}'(n_1,k, m_1)\times  \left[0, \left(\frac{(n,m)-1}{2}\right)\right]^k\rightarrow \mathcal{S}'(n,k, m),\]
    which restricts to a bijection
    \[\mathcal{S}_I'(n_1,k, m_1)\times  \left[0, \left(\frac{(n,m)-1}{2}\right)\right]^k\rightarrow \mathcal{S}_I'(n,k, m).\]
    Therefore, we conclude that \[\begin{split} & \# \mathcal{S}'(n,k, m)=\left(\frac{(n,m)+1}{2}\right)^k \# \mathcal{S}'(n_1,k, m_1), \\ 
    & \# \mathcal{S}_I'(n,k, m)=\left(\frac{(n,m)+1}{2}\right)^k \# \mathcal{S}_I'(n_1,k, m_1).
    \end{split}\]
\end{proof}
\begin{proposition}\label{prop 4.6}
    For $m\in [1, n-1]_{\Z}$, and $I\subseteq [-k,k]$ be a subset. Then the following assertions hold.
    
    \begin{enumerate}
        \item There is a surjection
    \[\Phi: \Omega_{n_1}'\rightarrow \mathcal{S}'(n_1,k,m_1)\]whose fibres all have cardinality $k!$. 
    \item The map $\Phi$ restricts to a map
    \[\Phi_{I}: \Omega_{n_1}'(I)\rightarrow \mathcal{S}_I'(n_1,k,m_1), \]whose fibers also have cardinality $k!$.
    \end{enumerate}
\end{proposition}
\begin{proof}
    We set $d:=(m,n)$, $m_1:=m/d$ and $n_1:=n/d$. Consider a point $x=(x_1, \dots, x_k)\in \Omega_{n_1}'$. By definition, there exist distinct integers $b_i\in [0, \frac{n_1-1}{2}]$ such that $x_i=b_i/n_1$. Let $c_i\in [0, n_1-1]$ be such that $m_1c_i\equiv b_i\mod{n_1}$ for all $i$. Then, set 
    \[a_i'=\begin{cases} c_i &\text{ if } c_i\in [0, \frac{n_1-1}{2}] \\
    (n_1-c_i) &\text{ if } c_i\in [ \frac{n_1+1}{2}, n_1-1] .
    \end{cases}\]
    Note that $a_i'\in [0, \frac{n_1-1}{2}]$ and that $\kappa\left(\frac{m_1a_i'}{n_1}\right)=\frac{m_1c_i}{n_1}=\frac{b_i}{n_1}$. There is a tuple $\vec{a}=(a_1, \dots, a_k)\in \mathcal{S}'(n_1, k, m_1)$ and a permutation $\sigma\in S_k$ such that $(a_{\sigma(1)}, \dots, a_{\sigma(k)})=(a_1', \dots, a_k')$. Thus, there is a unique permutation $\sigma$ such that 
    \[(x_1, \dots, x_k)=\left(\kappa\left(\frac{m_1a_{\sigma(1)}}{n_1}\right), \dots, \kappa\left(\frac{m_1 a_{\sigma(i)}}{n_1}\right), \dots, \kappa\left(\frac{m_1 a_{\sigma(k)}}{n_1}\right)\right).\] We map the tuple $(x_1, \dots, x_k)$ to $(a_1, \dots, a_k)$. Consider the action of $S_k$ on $\Omega_{n_1}'$ via permutation of coordinates. The fibres of this map have cardinality $k!=\# S_k$. By Lemma \ref{lem:interval}, it is easy to see that  $\Phi^{-1}(\mathcal{S}_I'(n_1, k, m_1))=\Omega_{n_1}'(I)$, and the second assertion follows. 
\end{proof}

\begin{lemma}\label{lemma 3.7}
Let $m\in[1, n-1]_{\Z}$ and $I$ be an interval contained in $[-k, k]$. Then we have that 
\[\begin{split} & \# \mathcal{S}'(n, k, m)=\left(\frac{(n,m)+1}{2}\right)^k\frac{\#\Omega_{n_1}'}{k!},\\ 
& \# \mathcal{S}'_I(n, k, m)=\left(\frac{(n,m)+1}{2}\right)^k\frac{\#\Omega_{n_1}'(I)}{k!}. \end{split}\]
\end{lemma}
\begin{proof}
   Lemma \ref{lemma 4.5} asserts that the following relations hold  \[\begin{split} & \# \mathcal{S}'(n,k, m)=\left(\frac{(n,m)+1}{2}\right)^k \# \mathcal{S}'(n_1,k, m_1), \\ 
     & \# \mathcal{S}_I'(n,k, m)=\left(\frac{(n,m)+1}{2}\right)^k \# \mathcal{S}_I'(n_1,k, m_1). \\ \end{split}\]

    By Proposition \ref{prop 4.6}, we have that
    \[\begin{split}
    &\# \mathcal{S}'(n_1,k, m_1)=\frac{\# \Omega_{n_1}}{k!},\\
    &\# \mathcal{S}_I'(n_1,k, m_1)=\frac{\# \Omega_{n_1}(I)}{k!},
    \end{split}\]
    and combining the above relations, the result follows.
\end{proof}

\begin{corollary}\label{size of S' lemma}
    Let $m\in [1, n-1]$, we have that 
    \[\begin{split} & \# \mathcal{S}'(n, k, m)=\left(\frac{(n,m)+1}{2}\right)^k\frac{\#\Omega_{n_1}}{k!}+O_k((n,m)n^{k-1}),\\ 
& \# \mathcal{S}'_I(n, k, m)=\left(\frac{(n,m)+1}{2}\right)^k\frac{\#\Omega_{n_1}(I)}{k!}+O_k((n,m)n^{k-1}). \end{split}\]
\end{corollary}
\begin{proof}
    Let $\Omega_{n_1}'':=\Omega_{n_1}\backslash \Omega_{n_1}'$ and $\Omega_{n_1}''(I):=\Omega_{n_1}(I)\backslash \Omega_{n_1}'(I)$. The set $\Omega_{n_1}''$ consists of all tuples $\left(x_1, \dots, x_k\right)$ such that for some $i<j$, $x_i=x_j$. Clearly, the cardinality of this set can be bounded as follows
\[\# \Omega_{n_1}''\leq \binom{k}{2}\left(\frac{n_1-1}{2}\right)^{k-1}.\]

It follows from Lemma \ref{lemma 3.7} that 
\[\begin{split} & \# \mathcal{S}'(n, k, m)=\left(\frac{(n,m)+1}{2}\right)^k\frac{\#\Omega_{n_1}}{k!}-\left(\frac{(n,m)+1}{2}\right)^k\frac{\#\Omega_{n_1}''}{k!},\\ 
& \# \mathcal{S}'_I(n, k, m)=\left(\frac{(n,m)+1}{2}\right)^k\frac{\#\Omega_{n_1}(I)}{k!}-\left(\frac{(n,m)+1}{2}\right)^k\frac{\#\Omega_{n_1}''(I)}{k!}. \end{split}\]

Therefore, the error term can be bounded as follows 
\[\left(\frac{(n,m)+1}{2}\right)^k\frac{\# \Omega_{n_1}''}{k!}\leq \frac{1}{k!}\binom{k}{2}\left(\frac{n_1-1}{2}\right)^{k-1}\left(\frac{(n,m)+1}{2}\right)^k=O_k((n,m)n^{k-1}),\] and the result follows.
\end{proof}

\section{Proof of the main theorem}

In the previous section, we obtained a geometric interpretation for the quantities $\#\mathcal{S}'(n, k, m)$ and $\#\mathcal{S}'_I(n, k, m)$, in terms of the quantities $\#\Omega_{n_1}'$ and $\#\Omega_{n_1}'(I)$ respectively. In order to effectively bound $\#\Omega_{n}'(I)$, it suffices bound $\#\Omega_{n}(I)$.

\begin{lemma}\label{shift lemma}
    Suppose that $I=[c,d]$ is an interval contained in $[-k, k]$, and suppose that $(\frac{a_1}{n}, \dots, \frac{a_k}{n})\in B_k(I)$. Then, $(\frac{a_1+1}{n}, \frac{a_2+1}{n}, \dots, \frac{a_k+1}{n})\in B_k\left(\left[c-\frac{2\pi k}{n},d+\frac{2\pi k}{n}\right]\right)$.
\end{lemma}

\begin{proof}
    It suffices to prove that \[\sum_{i=1}^k \cos\left(\frac{2\pi a_i}{n}\right) \in [c,d]\Longrightarrow \sum_{i=1}^k \cos\left(\frac{2\pi (a_i+1)}{n}\right) \in \left[c-\frac{2\pi k}{n},d+\frac{2\pi k}{n}\right].\]
    To see the above, notice that for $0\leq x\leq \pi$ and $\epsilon>0$ such that $x+\epsilon\leq \pi$, it follows from the mean value theorem that 
    \[\op{cos}(x)-\epsilon\leq \cos(x+\epsilon)\leq \op{cos}(x)+\epsilon. \]
\end{proof}

\begin{proposition}\label{size of Omega_n I prop}
    Let $I=[c,d]$ be an interval contained in $[-k,k]$. With respect to notation above, 
    \[\# \Omega_n(I)= n^k \op{Vol}(B_k(I))+O_k(n^{k-1}).\]
\end{proposition}
\begin{proof}
For a given element $(\frac{a_1}{n}, \dots, \frac{a_n}{n})\in \Omega_{n}(I)$, let $\mathcal{B}_{(a_1,\dots,a_k)}$ denote the box
\[\mathcal{B}_{(a_1,\dots,a_k)}= \left\{(x_1,\dots,x_k)\left|\frac{a_j}{n}\leq x_j \leq \frac{a_j+1}{n}\right.\right\}.\] It is clear that $\op{Vol}(\mathcal{B}_{(a_1,\dots,a_k)})=\frac{1}{n^k}$. Now we consider the set $U_{k,n}$ which is the union of  $\mathcal{B}_{(a_1,\dots,a_k)}$, where $(\frac{a_1}{n}, \dots, \frac{a_n}{n})\in \Omega_{n}(I)$. This is a set containing $B_k(I)$. We have 
\[\op{Vol} (B_k(I)) \leq \frac{\# \Omega_{n}(I)}{n^k} =\op{Vol} (U_{k,n}).\]
From Lemma \ref{shift lemma}, it follows that
\[U_{k,n}\subseteq B_k\left(\left[c-\frac{2\pi k}{n},d+\frac{2\pi k}{n}\right]\right).\]

Thus we conclude that 
\[\begin{split}n^k\op{Vol} (B_k(I)) & \leq \# \Omega_{n}(I) \\ 
& \leq n^k\op{Vol} \left(B_k\left(\left[c-\frac{2\pi k}{n},d+\frac{2\pi k}{n}\right]\right)\right)  \\ 
& = n^k\op{Vol} (B_k(I)) \\ + & n^k\op{Vol} \left(B_k\left(\left[c-\frac{2\pi k}{n},c\right]\right)\right)+n^k\op{Vol} \left(B_k\left(\left[d,d+\frac{2\pi k}{n}\right]\right)\right).\end{split}.\]
To estimate the error term, it suffices to bound 
\[\left(\frac{n}{2}\right)^k \int^c_{c-\frac{2\pi k}{n}} \delta^{(k)}(u) du \text{ and }\left(\frac{n}{2}\right)^k \int^{d+\frac{2\pi k}{n}}_{d} \delta^{(k)}(u) du.\]

Note that the support of $\delta$ is contained in $[-1, 1]$, and hence the support of $\delta^{(k)}$ is contained in $[-k, k]$. Therefore, if $c=-k$, then, $\int^c_{c-\frac{2\pi k}{n}} \delta^{(k)}(u) du=0$. Without loss of generality, assume that $c\in (-k, k)$, and let $c'\in (-k, c)$. Set $M:=\op{sup}\{|\delta^{(k)}(u)|\mid u\in [c', c]\}$. Note that for $n>\frac{2\pi k}{(c-c')}$, 
\[\int^c_{c-\frac{2\pi k}{n}} \delta^{(k)}(u) du\leq \frac{2\pi k M}{n}. \]
Therefore, we have shown that 
\[\left(\frac{n}{2}\right)^k \int^c_{c-\frac{2\pi k}{n}} \delta^{(k)}(u) du=O_k(n^{k-1}).\] The same reasoning shows that 
\[\left(\frac{n}{2}\right)^k \int^{d+\frac{2\pi k}{n}}_{d} \delta^{(k)}(u) du=O_k(n^{k-1}).\]
Therefore, we have shown that
\[\# \Omega_{n}(I) =n^k\op{Vol} (B_k(I))+O_{k}\left(n^{k-1}\right).\]

\end{proof}

\begin{proposition}\label{prop 3.11}
    With respect to notation above, 

    \[\# \mathcal{S}'_I(n, k, m)=\# \mathcal{S}'(n, k, m) 2^k\op{Vol}(B_k(I))+O_k((n,m)n^{k-1}).\]
\end{proposition}

\begin{proof} We have
    \[\begin{split}\# \mathcal{S}'_I(n, k, m)= &\left(\frac{(n,m)+1}{2}\right)^k\frac{\#\Omega_{n_1}(I)}{k!}+O_k((n,m)n^{k-1}),\\ 
    =& \left(\frac{(n,m)+1}{2}\right)^k \frac{n_1^k}{k!} \op{Vol}(B_k(I))+O_k((n,m)n^{k-1}) \\ 
    =& \left(\frac{(n,m)+1}{2}\right)^k \#\Omega_{n_1} \frac{2^k}{k!}\op{Vol}(B_k(I))+O_k((n,m)n^{k-1})\\
    =& \# \mathcal{S}'(n, k, m) 2^k\op{Vol}(B_k(I))+O_k((n,m)n^{k-1}). \end{split}\]
    For the above set of inequalities, we use Corollary \ref{size of S' lemma} and Proposition \ref{size of Omega_n I prop}.
\end{proof}

\begin{proposition}\label{prop 3.12}
    With respect to notation above, 

    \[\# \mathcal{S}'_I(n, k)=\# \mathcal{S}'(n, k) 2^k\op{Vol}(B_k(I))+O_k(n^{k+\epsilon}).\]
\end{proposition}
\begin{proof}
 Note that $\mathcal{S}'(n, k, 0)$ is empty. By an argument similar to the one employed in Lemma \ref{bounding S'' lemma},
 \[\begin{split}\#\mathcal{S}_I'(n, k) &= \sum_{m=1}^{n-1} \#\mathcal{S}_I'(n, k, m),\\
    &= \sum_{m=1}^{n-1} \#\mathcal{S}'(n, k, m) 2^k\op{Vol}(B_k(I))+ O_k\left( n^{k-1}\sum_{m=1}^{n-1} (n,m) \right),\\
          &=
          \# \mathcal{S}'(n, k) 2^k\op{Vol}(B_k(I))+O_k\left(n^{k-1}\sum_{d|n} d \varphi\left(\frac{n}{d}\right)\right),\\
          &=
          \# \mathcal{S}'(n, k) 2^k\op{Vol}(B_k(I))+O_k(n^{k+\epsilon}),
     \end{split}
\]
where in the second line, we have applied Proposition \ref{prop 3.11}.
    \end{proof}

\begin{theorem}\label{main theorem 5.5}
For $k\geq 1$ fixed, and $I=[c,d]\subseteq [-k,k]$,
\[\op{Prob}_I(n,k)=\int_{c}^d \delta^{(k)}(u) du+O_k(n^{-1+\epsilon}).\]
\end{theorem}

\begin{proof}
  Recall from \eqref{prob sum eqn} that 
    \[\op{Prob}_I(n, k)=\frac{1}{n}\sum_{m=0}^{n-1} \op{Prob}_I(n, k, m),\]
    where 
  \[\op{Prob}_I(n, k, m)=\frac{\#\mathcal{S}_I(n, k, m)}{\#\mathcal{S}(n, k, m)}=\frac{\#\mathcal{S}_I(n, k, m)}{\binom{\frac{n-1}{2}}{k}}.\]
  Note that $\# \mathcal{S}_I(n, k, m)=\# \mathcal{S}_I'(n, k, m)+\# \mathcal{S}_I''(n, k, m)$, and 
  therefore, 
  \[\begin{split}\op{Prob}_I(n, k)=& \frac{1}{n}\sum_{m=0}^{n-1} \op{Prob}_I(n, k, m) \\ 
  = & \frac{1}{n}\sum_{m=0}^{n-1} \frac{\mathcal{S}_I'(n, k, m)}{\binom{\frac{n-1}{2}}{k}}+\frac{\mathcal{S}_I''(n, k)}{n\binom{\frac{n-1}{2}}{k}}.\\ \end{split}\]
  
  By Lemma \ref{bounding S'' lemma}, we find that 
  \[\mathcal{S}_I''(n, k)\leq \mathcal{S}''(n,  k)=O_k(d(n)n^k)=O_k(n^{k+\epsilon}),\] and therefore, 
  \[\begin{split}\op{Prob}_I(n, k)=& \frac{1}{n} \frac{\mathcal{S}_I'(n, k)}{\binom{\frac{n-1}{2}}{k}}+O_k(n^{-1+\epsilon}).
  \end{split}\]
  By Proposition \ref{prop 3.12},
  \[\begin{split}\# \mathcal{S}'_I(n, k)=& \# \mathcal{S}'(n, k) 2^k\op{Vol}(B_k(I))+O_k(n^{k+\epsilon}),\\
  =& \# \mathcal{S}(n, k) 2^k\op{Vol}(B_k(I))+O_k(n^{k+\epsilon}),\\
  =&  n\binom{\frac{n-1}{2}}{k} 2^k\op{Vol}(B_k(I))+O_k(n^{k+\epsilon}),\\
  \end{split}\]
where we recall that for fixed $k$, $\op{Vol}(B_k(I))$ is bounded uniformly in $I$, which follows from Lemma \ref{lem:vol-computation}. Moreover,  by Lemma \ref{lem:vol-computation}, 
  \[\begin{split}\op{Prob}_I(n, k)=& 2^k\op{Vol}(B_k(I)) +O_k(n^{-1+\epsilon}),\\
  = & \int_{c}^d \delta^{(k)}(u) du+O_k(n^{-1+\epsilon}).
  \end{split}\]
  This completes the proof.
\end{proof}

\begin{proof}[Proof of Theorem \ref{main thm}]
    It follows from \eqref{probI} that
   \[\op{Prob}_I(n, k)= \frac{\#\{(X, \alpha)\mid X\in \cF_r, \alpha\in \op{Sp}(X), h(X)=n, \alpha\in J\} }{\#\{(X, \alpha)\mid X\in \cF_r, \alpha\in \op{Sp}(X), h(X)=n\}}.\]
By Theorem \ref{main theorem 5.5}, 
\[\op{Prob}_I(n,k)=\int_{c}^d \delta^{(k)}(u) du+O_k(n^{-1+\epsilon}),\] which proves the result.
\end{proof}



\end{document}